\documentclass[11pt, a4paper]{amsart}
\usepackage{amssymb}
\usepackage{amsmath}
\usepackage{amssymb}
\usepackage{amsthm}
\usepackage{bbm}
\usepackage{bm}
\usepackage{mathrsfs}
\usepackage[T1]{fontenc}
\usepackage[usenames,dvipsnames,svgnames,table]{xcolor}
\usepackage{esint}

\usepackage{enumitem}
\setenumerate{label={\rm (\alph{*})}}
\usepackage[colorlinks=true,citecolor=Mahogany,urlcolor=blue,
linkcolor=RoyalBlue]{hyperref}

\usepackage{graphicx}
\usepackage{tikz}

\usepackage{geometry}
\numberwithin{equation}{section}

\newtheorem{theorem}{Theorem}[section]
\newtheorem{lemma}[theorem]{Lemma}
\newtheorem{proposition}[theorem]{Proposition}

\newtheorem{corollary}[theorem]{Corollary}

\newtheorem{definition}[theorem]{Definition}

\newtheorem{remark}[theorem]{Remark}

\numberwithin{equation}{section}

\usepackage{booktabs}

\normalsize

\providecommand{\comment}[1]{\vskip.3cm
\fbox{%
\parbox{0.93\linewidth}{\footnotesize #1}}
\vskip.3cm}

\def\XXint#1#2#3{{\setbox0=\hbox{$#1{#2#3}{\int}$ }
\vcenter{\hbox{$#2#3$ }}\kern-.6\wd0}}

\newcommand{\mres}{\mathbin{\vrule height 1.6ex depth 0pt width
0.13ex\vrule height 0.13ex depth 0pt width 1.3ex}}

\newcommand{\bd}{\operatorname{BD}}
\newcommand{\bv}{\operatorname{BV}}

\newcommand{\dif}{\operatorname{d}\!}

\newcommand{\N}{\mathbb{N}}

\newcommand{\R}{\mathbb{R}}
\newcommand{\A}{\mathbb{A}}

\newcommand{\locc}{\operatorname{loc}}

\newcommand{\ball}{\operatorname{B}}

\newcommand{\sobo}{\operatorname{W}}

\newcommand{\lebe}{\operatorname{L}}

\newcommand{\hold}{\operatorname{C}}

\newcommand{\D}{\operatorname{D}\!}

\renewcommand{\leq}{\leqslant}

\newcommand{\lin}{\mathscr{L}}

\usepackage{tikz-cd}
\begin{document}
\title[$\lebe^{n/(n-1)}$--differentiability of $\bv^\A$--maps]{On critical $\lebe^p$--differentiability of $\bd$--maps}
\author[F. Gmeineder]{Franz Gmeineder}
\author[B. Raita]{Bogdan Raita}
\address{Franz Gmeineder: Mathematisches Institut, Universitat Bonn, Endenicher Allee 60, 53115 Bonn, Germany. Email: fgmeined@math.uni-bonn.de}
\address{Bogdan Raita: Andrew Wiles Building, University of Oxford, Woodstock Rd, Oxford OX2 6GG, United Kingdom. Email: raita@maths.ox.ac.uk}
\begin{abstract}
We prove that functions of locally bounded deformation on $\R^n$ are $\lebe^{\frac{n}{n-1}}$--differentiable $\mathcal{L}^n$--almost everywhere. More generally, we show that this critical $\lebe^p$--differentiability result holds for functions of locally bounded $\A$--variation, provided that the first order, homogeneous differential operator $\A$ has finite dimensional null--space.
\end{abstract}
\maketitle
\providecommand{\commentfranz}[1]{\comment{\textcolor{blue}{Franz:
      #1}}}
\providecommand{\commentbogdan}[1]{\comment{\textcolor{blue}{Bogdan:
        #1}}}
\section{Introduction}\label{sec:intro}
Approximate differentiability properties of weakly differentiable functions are reasonably well understood. Namely, it is well--known that maps in $\sobo_{\locc}^{1,p}(\R^n,\R^N)$ are $\lebe^{p^*}$--differentiable $\mathcal{L}^n$--a.e. in $\R^n$, where $1\leq p<n$, $p^*:=np/(n-p)$ (see, e.g., \cite[Thm~6.2]{EG}). We recall that a map $u\colon\R^n\rightarrow\R^N$ is $\lebe^q$--approximately differentiable at $x\in\R^n$ if and only if there exists a matrix $M\in\R^{N\times n}$ such that
\begin{align*}
\left(\fint_{\ball_r(x)}|u(y)-u(x)-M(y-x)|^q\dif y\right)^\frac{1}{q}=o(r)
\end{align*}
as $r\downarrow0$, whence, in particular, $u$ is approximately differentiable at $x$ with approximate gradient $M$ (see Section \ref{sec:prel} for precise definitions). For $p=1$ one can show in addition that maps $u\in\bv_{\locc}(\R^n,\R^N)$ are $\lebe^{1^*}$--differentiable $\mathcal{L}^n$--a.e. with the approximate gradient equal $\mathcal{L}^n$--a.e. to the absolutely continuous part of $\D u$ (\cite[Thm.~6.1,~6.4]{EG}). It is natural to ask a similar question of the space $\bd(\R^n)$ of functions of bounded deformation, i.e., of $\lebe^1(\R^n,\R^n)$--maps $u$ such that the symmetric part $\mathcal{E}u$ of their distributional gradient is a bounded measure. The situation in this case is significantly more complicated, since, for example, we have $\bv(\R^n,\R^n)\subsetneq\bd(\R^n)$ by the so--called Ornstein's Non--inequality \cite{CFM,KK,Ornstein}; equivalently, there are maps $u\in\bd(\R^n)$ for which the full distributional gradient $\D u$ is not a Radon measure, so one cannot easily retrieve the approximate gradient of $u$ from the absolutely continuous part of $\mathcal{E}u$ with respect to $\mathcal{L}^n$. It is however possible to recover $u$ from $\mathcal{E}u$ via convolution with a $(1-n)$--homogeneous kernel (cp. Lemma \ref{lem:representation}). \textsc{Haj\l{}asz} used this observation and a
Marcinkiewicz--type characterisation of approximate differentiability to show approximate differentiability
$\mathcal{L}^n$--a.e. of $\bd$--functions (\cite[Cor.~1]{Haj}). This result was improved in \cite[Thm.~7.4]{ACDM} to $\lebe^1$--differentiability $\mathcal{L}^n$--a.e. by \textsc{Ambrosio}, \textsc{Coscia}, and \textsc{Dal Maso}, using the precise Korn--Poincar\'e Inequality of \textsc{Kohn} \cite{bob}. It was only recently when \textsc{Alberti}, \textsc{Bianchini}, and \textsc{Crippa} generalized the approach in \cite{Haj}, obtaining $\lebe^q$--differentiability of $\bd$--maps for $1\leq q<1^*$ (see \cite[Thm.~3.4,~Prop.~4.3]{ABC}). It is, however, unclear whether the critical exponent $q=1^*$ can be reached using the Calder\'on--Zygmund--type approach in \cite{ABC}.

In the present paper, we settle the question in \cite[Rk.~4.5.(v)]{ABC} of optimal differentiability of $\bd$--maps in the positive (see Corollary \ref{cor:BD}). Although reminiscent of the elaborate estimates in \cite[Sec.~7]{ACDM}, our proof is rather straightforward. The key observation is to replace \textsc{Kohn}'s Poincar\'e--Korn Inequality with the more abstract Korn--Sobolev Inequality due to \textsc{Strang} and \textsc{Temam} \cite[Prop.~2.4]{ST}, combined with ideas developed recently by the authors in \cite{GR}. In fact, we shall prove $\lebe^{n/(n-1)}$--differentiability of maps of bounded $\A$--variation (as introduced in \cite[Sec.~2.2]{BDG}), provided that $\A$ has finite dimensional null--space.

To formally state our main result, we pause to introduce some terminology and notation. Let $\A$ be a linear, first order, homogeneous differential operator with constant coefficients on $\R^n$ from $V$ to $W$, i.e.,
\begin{align}\label{eq:A}
\A u=\sum_{j=1} A_j\partial_j u,\qquad u\colon\R^{n}\to V,
\end{align} 
where $A_j\in\lin(V,W)$ are fixed linear mappings between two finite dimensional real vector spaces $V$ and $W$. For an open set $\Omega\subset\R^n$, we define $\bv^\A(\Omega)$ as the space of $u\in\lebe^1(\Omega,V)$ such that $\A u$ is a $W$--valued Radon measure. We say that $\A$ has \emph{FDN} (finite dimensional null--space) if the vector space $\{u\in\mathscr{D}^\prime(\R^n,V)\colon\A u=0\}$ is finite dimensional. Using the main result in \cite[Thm.~1.1]{GR}, we will prove that FDN is sufficient to obtain a Korn--Sobolev--type inequality
\begin{align}\label{eq:poinc-sob}
\left(\fint_{\ball_r}|u-\pi_{\ball_r}u|^{\frac{n}{n-1}}\dif x\right)^\frac{n-1}{n}\leq c r\fint_{\ball_r}|\A u|\dif x
\end{align}
for all $u\in\hold^\infty(\bar{\ball}_r,V)$. Here $\pi$ denotes a suitable bounded projection on the null--space of $\A$, as described in \cite[Sec.~3.1]{BDG}. This is our main ingredient to prove the following:
\begin{theorem}\label{thm:main}
Let $\A$ as in \eqref{eq:A} have FDN, $u\in\bv^\A_{\locc}(\R^n)$. Then $u$ is $\lebe^{n/(n-1)}$--differentiable at $x$ for $\mathcal{L}^n$--a.e. $x\in\R^n$.
\end{theorem}
Our example of interest is $\bd:=\bv^\mathcal{E}$, where $\mathcal{E}u:=\left(\D u+(\D u)^\mathsf{T}\right)/2$ for $u\colon\R^n\rightarrow\R^n$. It is well known that the null--space of $\mathcal{E}$ consists of rigid motions, i.e., affine maps of anti--symmetric gradient. In particular, $\mathcal{E}$ has FDN.
\begin{corollary}\label{cor:BD}
Let $u\in\bd_{\locc}(\R^n)$. Then $u$ is $\lebe^{n/(n-1)}$--differentiable $\mathcal{L}^n$--a.e.
\end{corollary}
This paper is organized as follows: In Section \ref{sec:prel} we collect some notation and definitions, mainly those of approximate and  $\lebe^p$--differentiability, present the main result in \cite{ABC}, collect a few results on $\A$--weakly differentiable functions from \cite{BDG,GR}, and prove the inequality \eqref{eq:poinc-sob}. In Section \ref{sec:proof} we give a brief proof of Theorem \ref{thm:main}.

\subsection*{Acknowledgement} The authors wish to thank Jan Kristensen for reading a preliminary version of the paper. The second author was supported by Engineering and
Physical Sciences Research Council Award EP/L015811/1.
\section{Preliminaries}\label{sec:prel}
An operator $\A$ as in \eqref{eq:A} can also be seen as $\A u=A(\D u)$ for $u\colon\R^n\rightarrow V$, where $A\in\lin(V\otimes\R^n,W)$. We recall that such an operator has a Fourier symbol map
\begin{align*}
\A[\xi]v=\sum_{j=1}^n \xi_j A_jv,
\end{align*}
defined for $\xi\in\R^n$ and $v\in V$. An operator $\A$ is said to be \emph{elliptic} if and only if for all non--zero $\xi$, the maps $\A[\xi]\in\lin(V,W)$ are injective. By considering the maps
\begin{align*}
u_f(x):=f(x\cdot\xi)v
\end{align*}
for functions $f\in\hold^1(\R)$, it is easy to see that if $\A$ has FDN, then $\A$ is necessarily elliptic. Ellipticity is in fact equivalent with one--sided invertibility of $\A$ in Fourier space; more precisely, the equation $\A u=f$ can be uniquely solved for $u\in\mathscr{S}(\R^n,V)$ whenever $f\in\mathscr{S}(\R^n,W)\cap\mathrm{im}\A$. One has:
\begin{lemma}\label{lem:representation}
Let $\A$ be elliptic. There exists a convolution kernel $K^\A\in\hold^\infty(\R^n\setminus\{0\},\lin(W,V))$ which is $(1-n)$--homogeneous such that $u=K^\A*\A u$ for all $u\in\mathscr{S}(\R^n,V)$. 
\end{lemma}
For a proof of this fact, see, e.g., \cite[Lem.~2.1]{GR}. We next define, for open $\Omega\subset\R^n$ (often a ball $\ball_r(x)$), the space
\begin{align*}
\bv^\A(\Omega):=\{u\in\lebe^1(\Omega,V)\colon\A u\in\mathcal{M}(\Omega,W)\}
\end{align*}
of maps of bounded $\A$--variation, which is a Banach space under the obvious norm. By the Radon--Nikodym Theorem $\A u$ has the decomposition
\begin{align*}
\A u=\A^{ac} u \mathcal{L}^n\mres\Omega+\A^su:=\dfrac{\dif\A u}{\dif\mathcal{L}^n}\mathcal{L}^n\mres\Omega+\dfrac{\dif\A^su}{\dif|\A^su|}|\A^su|
\end{align*}
with respect to $\mathcal{L}^n$. Here $|\cdot|$ denotes the total variation semi--norm. We next see that ellipticity of $\A$ implies sub--critical $\lebe^p$--differentiability. We denote averaged integrals by $\fint_\Omega:=\mathcal{L}^n(\Omega)^{-1}\int_\Omega$ or by $(\cdot)_{x,r}$ if $\Omega=\ball_r(x)$, the ball of radius $r>0$ centred at $x\in\R^n$.
\begin{definition}
A measurable map $u\colon\R^n\rightarrow V$ is said to be 
\begin{itemize}
\item approximately differentiable at $x\in\R^n$ if there exists a matrix $M\in V\otimes\R^n$ such that
\begin{align*}
\underset{y\rightarrow x}{\mathrm{ap}\lim}\,\dfrac{|u(y)-u(x)-M(y-x)|}{|y-x|}=0;
\end{align*}
\item $\lebe^p$--differentiable at $x\in\R^n$, $1\leq p<\infty$ if there exists a matrix $M\in V\otimes\R^n$ such that
\begin{align*}
\left(\fint_{\ball_r(x)}|u(y)-u(x)-M(y-x)|^p\dif y\right)^\frac{1}{p}=o(r)
\end{align*}
as $r\downarrow0$.
\end{itemize}
We say that $\nabla u(x):=M$ is the approximate gradient of $u$ at $x$.
\end{definition}
We should also recall that 
\begin{align*}
v=\underset{y\rightarrow x}{\mathrm{ap}\lim}\,u(y)\iff \forall \varepsilon>0,\,\lim_{r\downarrow0}r^{-n}\mathcal{L}^n\left(\{y\in\ball_r(x)\colon|u(y)-v|>\varepsilon\}\right)=0,
\end{align*}
where $x\in\R^n$ and $u\colon\R^n\rightarrow V$ is measurable. In the terminology of \cite[Sec.~2.2]{ABC}, we can alternatively say that $u$ is $\lebe^p$--differentiable at $x$ if 
\begin{align}\label{eq:taylor}
u(y)=\nabla u(x)(y-x)+u(x)+R_x(y),
\end{align}
where $(|R_x|^p)_{x,r}=o(r^p)$ as $r\downarrow0$. We will refer to the decomposition \eqref{eq:taylor} as a first order $\lebe^p$--Taylor expansion of $u$ about $x$.
\begin{theorem}[{\cite[Thm.~3.4]{ABC}}]\label{thm:ABC_main}
Let $K\in\hold^2(\R^n\setminus\{0\})$ be $(1-n)$--homogeneous, and $\mu\in\mathcal{M}(\R^n)$ be a bounded measure. Then $u:=K*\mu$ is $\lebe^p$--differentiable $\mathcal{L}^n$--a.e. for all $1\leq p<n/(n-1)$.
\end{theorem}
As a consequence of Lemma \ref{lem:representation} and Theorem \ref{thm:ABC_main}, we have that if $\A$ is elliptic, then maps in $\bv^\A(\R^n)$ are $\lebe^p$--differentiable $\mathcal{L}^n$--a.e. for $1\leq p<n/(n-1)$ (cp. Lemma \ref{lem:sub_crit}). Ellipticity, however, is insufficient to reach the critical exponent.  In Theorem \ref{thm:main}, we show that FDN is a sufficient condition for the critical $\lebe^{n/(n-1)}$--differentiability.
The following is essentially proved in \cite{Smith}, and is discussed at length in \cite{BDG,GR}. We will, however, sketch an elementary proof for the interested reader.
\begin{lemma}\label{lem:FDNpoly}
Let $\A$ as in \eqref{eq:A} have FDN. Then there exists $l\in\N$ such that null--space elements of $\A$ are polynomials of degree at most $l$.
\end{lemma}
\begin{proof}[Sketch]
One can show by standard arguments that if $\A$ is elliptic and $\A u=0$ in $\mathscr{D}^\prime(\R^n,V)$, then $u$ is in fact analytic. If $u$ is not a polynomial, then one can write $u$ as an infinite sum of homogeneous polynomials and identify coefficients, thereby obtaining infinitely many linearly independent (homogeneous) polynomials in the null--space of $\A$. Then the kernel consists of polynomials, which must have a maximal degree, otherwise $\A$ fails to have FDN.
\end{proof}
We next provide a Sobolev--Poincar\'{e}--type inequality which, in the $\A$--setting, follows from the recent work \cite{GR} and is the main ingredient in the proof of Theorem~\ref{thm:main}. Following \cite[Sec.~3.1]{BDG}, we define for $\A$ with FDN, $\pi_B\colon\hold^\infty\cap\bv^\A(\ball)\rightarrow\ker\A\cap\lebe^2(\ball,V)$ as the $\lebe^2$--projection onto $\ker\A$.
\begin{proposition}[Poincar\'e--Sobolev--type Inequality]\label{prop:poinc-sob}
Let $\A$ as in \eqref{eq:A} have FDN. Then \eqref{eq:poinc-sob} holds. Moreover, there exists $c>0$ such that
\begin{align*}
\left(\fint_{\ball_r(x)}|u-\pi_{\ball_r(x)}u|^{\frac{n}{n-1}}\dif y\right)^\frac{n-1}{n}\leq c r^{1-n}|\A u|(\overline{\ball_r(x)}).
\end{align*}
for all $u\in\bv^\A_{\locc}(\R^n)$, $x\in\R^n$, $r>0$.
\end{proposition}
\begin{proof}
By smooth approximation (\cite[Thm.~2.8]{BDG}), it suffices to prove \eqref{eq:poinc-sob}. Since $\pi_{\ball_r(x)}$ is linear, we can assume that $r=1$, $x=0$. The result then follows by scaling and translation. We abbreviate $\ball:=\ball_1(0)$. By \cite[Thm.~1.1]{GR} we have that
\begin{align*}
\left(\int_{\ball}|u-\pi_{\ball} u|^{\frac{n}{n-1}}\dif y\right)^\frac{n-1}{n}\leq c \left(\int_{\ball}|\A u|+|u-\pi_{\ball}u|\dif y\right)\leq c\int_{\ball}|\A u|\dif y,
\end{align*}
where for the second estimate we use the Poincar\'e--type inequality in \cite[Thm.~3.2]{BDG}. The proof is complete.
\end{proof}
We conclude this section with a simple technical Lemma:
\begin{lemma}
Let $l\in\N$. There exists a constant $c>0$ independent of any ball $\ball\subset\R^n$ such that
\begin{align}\label{eq:norm_equiv}
\sup_{y\in\ball}|P(y)|\leq c \fint_{\ball}|P(y)|\dif y
\end{align}
for any polynomial of degree at most $l$.
\end{lemma}
\begin{proof}
The space the polynomials of degree at most $l$ restricted on the unit ball is finite dimensional, hence the $\lebe^\infty$ and $\lebe^1$ norms are equivalent. In particular, \eqref{eq:norm_equiv} holds for $\ball=\ball_1(0)$. Consider $\ball:=\ball_r(x)$. Then
\begin{align*}
\sup_{y\in\ball}|P(y)|=\sup_{z\in\ball_1(0)}|P(x+rz)|\leq c\fint_{\ball_1(0)}|P(x+rz)|\dif z=c\fint_{\ball_r(x)}|P(y)|\dif y,
\end{align*}
since $P(x+r\,\cdot)$ are polynomials of degree at most $l$. The proof is complete.
\end{proof}
\section{Proof of Theorem \ref{thm:main}}\label{sec:proof}
We begin by proving sub--critical $\lebe^p$--differentiability of $u\in\bv^\A$ for elliptic $\A$ (cp. \cite[Thm.~5]{Haj}). We also provide a formula that enables us to retrieve the absolutely continuous part of $\A u$ from the approximate gradient. This formula respects the algebraic structure of $\A$, generalizing the result for $\bd$ in \cite[Rk.~7.5]{ACDM}.
\begin{lemma}\label{lem:sub_crit}
If $\A$ is elliptic, then any map $u\in\bv^\A(\R^n)$ is $\lebe^p$--differentiable $\mathcal{L}^n$--a.e. for all $1\leq p<n/(n-1)$. Moreover, we have that 
\begin{align}\label{eq:structure_Au}
\dfrac{\dif\A u}{\dif \mathcal{L}^n}(x)=A(\nabla u(x))
\end{align}
for $\mathcal{L}^n$--a.e $x\in\R^n$.
\end{lemma}
\begin{proof}
By Lemma \ref{lem:representation}, we can write the components $u_i=K^\A_{ij}*(\A u)_j$, where summation over repeated indices is adopted. We then note that $K^\A_{ij}$ satisfies the assumptions of Theorem \ref{thm:ABC_main}, hence each component $u_i$ is $\lebe^p$--differentiable $\mathcal{L}^n$--a.e. for $1\leq p<n/(n-1)$.

We next let $u\in\bv^\A(\R^n)$ and $x\in\R^n$ be a Lebesgue point of $u$ and $\A^{ac}u$, and also a point of $\lebe^1$--differentiability of $u$. We also consider a sequence $(\eta_\varepsilon)_{\varepsilon>0}$ of standard mollifiers, i.e., $\eta_1\in\hold^\infty_c(\ball_1(0))$ is radially symmetric and has integral equal to 1 and $\eta_\varepsilon(y)=\varepsilon^{-n}\eta_1(x/\varepsilon)$. Finally, we write $u_\varepsilon:=u*\eta_\varepsilon$ and employ the Taylor expansion \eqref{eq:taylor} to compute
\begin{align*}
\nabla u_\varepsilon(x)&=\int_{\ball_\varepsilon(x)} u(y)\otimes\nabla_x\eta_\varepsilon(x-y)\dif y\\
&=-\int_{\ball_\varepsilon(x)}\left(\nabla u(x)(y-x)+u(x)+R_x(y)\right)\otimes\nabla_y\eta_\varepsilon(y-x)\dif y\\
&=\int_{\ball_\varepsilon(x)}\eta_\varepsilon(y-x)\nabla u(x)\dif y-\int_{\ball_\varepsilon(x)}R_x(y)\otimes\nabla_y\eta_\varepsilon(y-x)\dif y\\
&=\nabla u(x)+\int_{\ball_\varepsilon(x)}R_x(y)\otimes\nabla_x\eta_\varepsilon(x-y)\dif y,
\end{align*}
where we used integration by parts to establish the third equality. Since \begin{align*}
\|\nabla_x\eta(x-\cdot)\|_\infty=\varepsilon^{-(n+1)}\|\nabla\eta_1\|_\infty,
\end{align*}
we have that $|\nabla u_\varepsilon(x)-\nabla u(x)|\leq c(n,\eta_1)\varepsilon^{-1}(|R_x|)_{x,\varepsilon}=o(1)$ as $x$ is a point of $\lebe^1$--differentiability of $u$. In particular, $\nabla u_\varepsilon\rightarrow\nabla u$ $\mathcal{L}^n$--a.e., so that $\A u_\varepsilon\rightarrow A(\nabla u)$ $\mathcal{L}^n$--a.e. To establish \eqref{eq:structure_Au}, we will show that $\A u_\varepsilon\rightarrow\A^{ac}u$ $\mathcal{L}^n$--a.e. Using only that $u$ is a distribution, one easily shows that $\A u_\varepsilon=\A u*\eta_\varepsilon$, so that
\begin{align*}
\A u_\varepsilon(x)-\A^{ac}u(x)&=\A^{ac}u*\eta_\varepsilon(x)-\A^{ac}u(x)+\A^s u*\eta_\varepsilon(x)\\
&=\int_{\ball_\varepsilon(x)}\eta_\varepsilon(x-y)\left(\A^{ac}u(y)-\A^{ac}u(x)\right)\dif y \\
&+\int_{\ball_\varepsilon(x)}\eta_\varepsilon(x-y)\dif \A^s u(y).
\end{align*}
Using the fact that $\|\eta_\varepsilon(x-\cdot)\|_\infty=\varepsilon^{-n}\|\eta_1\|_\infty$ and Lebesgue differentiation, the proof is complete.
\end{proof}
\begin{remark}[Insufficiency of ellipticity]\label{rk:ell_insuff}
Consider $v$ as in \cite[Prop.~4.2]{ABC} with $n=2$. One shows by direct computation that $v\in\bv^{\partial}(\R^2)$, where the Wirtinger derivative 
\begin{align*}
\partial u:=\dfrac{1}{2}\left(\begin{array}{c}
\partial_1 u_1+\partial_2 u_2\\
\partial_2 u_1 -\partial_1 u_2
\end{array}\right)
\end{align*}
is easily seen to be elliptic (computation). However, it is shown in \cite[Rk.~4.5.(iv)]{ABC} that there are maps $v\in\bv^\partial(\R^2)$ which are not $\lebe^2$--differentiable.
\end{remark}
In turn, the stronger FDN condition is sufficient for $\lebe^{1^*}$--differentiability:
\begin{proof}[Proof of Theorem \ref{thm:main}]
Let $u\in\bv^\A_{\locc}(\R^n)$ and $x\in\R^n$ that is a Lebesgue point of $\A u$ such that 
\begin{align}\label{eq:L1_diff}
\fint_{\ball_r(x)}|u(y)-u(x)-\nabla u(x)(y-x)|\dif y=o(r)
\end{align}
as $r\downarrow0$. By Lemma \ref{lem:sub_crit} for $p=1$, such points exist $\mathcal{L}^n$--a.e. Here $\nabla u(x)$ denotes the approximate gradient of $u$ at $x$. We define $v(y):=u(y)-u(x)-\nabla u(x)(y-x)$ for $y\in\R^n$. We aim to show that
\begin{align}\label{eq:sharp_diff}
\left(\fint_{\ball_r(x)}|v(y)|^{\frac{n}{n-1}}\dif y\right)^{\frac{n-1}{n}}=o(r)
\end{align}
as $r\downarrow0$. Firstly, we remark that the integral in \eqref{eq:sharp_diff} is well--defined for $r>0$, as $v$ is the sum of an affine and a $\bv^\A_{\locc}$--map; the latter is $\lebe^{n/(n-1)}_{\locc}$--integrable, e.g., by \cite[Thm.~1.1]{GR}. Next, we abbreviate $\pi_r v:=\pi_{\ball_r(x)}v$ and use Proposition \ref{prop:poinc-sob} to estimate:
\begin{align*}
\left(\fint_{\ball_r(x)}|v|^{1^*}\dif y\right)^{\frac{1}{1^*}}
&\leq\left(\fint_{\ball_r(x)}|v-\pi_r v|^{1^*}\dif y\right)^{\frac{1}{1^*}}
+\left(\fint_{\ball_r(x)}|\pi_r v|^{1^*}\dif y\right)^{\frac{1}{1^*}}\\
&\leq c r\dfrac{|\A v|(\overline{\ball_r(x)})}{r^n}+\left(\fint_{\ball_r(x)}|\pi_r v|^{\frac{n}{n-1}}\dif y\right)^{\frac{n-1}{n}}=:\textbf{I}_r+\textbf{II}_r.
\end{align*}
To deal with $\textbf{I}_r$, first note that $\A v=\A u-A(\nabla u(x))$ (the latter term is obtained by classical differentiation of an affine map). By \eqref{eq:structure_Au}, we obtain $\A v=\A u-\A^{ac} u(x)$, so $\textbf{I}_r=o(r)$ as $r\downarrow0$ by Lebesgue differentiation for Radon measures. To bound $\textbf{II}_r$, by Lemma \ref{lem:FDNpoly}, we can use \eqref{eq:norm_equiv} to get that
\begin{align*}
\left(\fint_{\ball_r(x)}|P|^{\frac{n}{n-1}}\dif y\right)^{\frac{n-1}{n}}\leq c \fint_{\ball_r(x)}|P|\dif y,
\end{align*}
so that we have $\textbf{II}_r\leq c(|\pi_r v|)_{x,r}$. We claim that 
\begin{align}\label{eq:L1_proj}
\fint_{\ball_r(x)}|\pi_r v|\dif y\leq c\fint_{\ball_r(x)}|v|\dif y,
\end{align}
which suffices to conclude by \eqref{eq:L1_diff}, and \eqref{eq:sharp_diff}. Though elementary and essentially present in \cite[Sec.~3.1]{BDG}, the proof of \eqref{eq:L1_proj} is delicate and we present a careful argument. We write 
\begin{align*}
\pi_r v=\sum_{j=1}^d \langle v,e^r_j\rangle e^r_j,
\end{align*}
where the inner product is taken in $\lebe^2$ and $\{e_j^r\}_{j=1}^d$ is a (finite) orthonormal basis of $\ker\A\cap\lebe^2(\ball_r(x),V)$. By \eqref{eq:norm_equiv} and Cauchy--Schwarz inequality we have
\begin{align*}
\sup_{y\in\ball_r(x)}|e_j^r(y)|\leq c \left(\fint_{\ball_r(x)}|e_j^r|^2\dif y\right)^{\frac{1}{2}}=cr^{-\frac{n}{2}},
\end{align*}
so that
\begin{align*}
\fint_{\ball_r(x)}|\pi_r v|\dif y\leq \sum_{j=1}^d \fint_{\ball_r(x)}\int_{\ball_r(x)}|v|\dif z\dif y \|e_j^r\|^2_{\lebe^\infty(\ball_r(x),V)}\leq c r^{-n}\int_{\ball_r(x)}|v|\dif z,
\end{align*}
which yields \eqref{eq:L1_proj} and concludes the proof.
\end{proof}

\end{document}